\documentclass[11pt]{article}
\usepackage{amsmath,amsthm,amssymb}
\usepackage[dvips]{graphicx}

\newcommand{\N}{\mathbb N}
\newcommand{\Z}{\mathbb Z}
\renewcommand{\P}{\mathbb P}
\newcommand{\R}{\mathbb R}
\newcommand{\C}{\mathbb C}
\newcommand{\h}{{\cal H}}
\newcommand{\T}{{\cal T}}
\newcommand{\U}{{\cal U}}
\newcommand{\CC}{{\cal C}}
\newcommand{\ii}{{\operatorname{i}}}
\newcommand{\dd}{{\operatorname{d}}}
\newcommand{\sn}{{\operatorname{sn}}}
\newcommand{\cn}{{\operatorname{cn}}}
\newcommand{\dn}{{\operatorname{dn}}}
\newcommand{\zn}{{\operatorname{zn}}}
\newcommand{\re}{{\operatorname{Re}}}
\newcommand{\im}{{\operatorname{Im}}}
\newcommand{\laK}{\lambda{K}}

\newcommand{\iK}{\ii{K}'}
\newcommand{\Tn}{\T_n^{-1}([-1,1])}
\newcommand{\tn}{\mathbf{T}_n}
\newcommand{\arc}{\widetilde{a_3a_4}}
\newcommand{\ov}{\overline}

\begin{document}


\title{Inverse Polynomial Images Consisting of an Interval and an Arc\footnote{published in: Computational Methods and Function Theory {\bf 9} (2009), 407--420.}}
\author{Klaus Schiefermayr\footnote{University of Applied Sciences Upper Austria, School of Engineering and Environmental Sciences, Stelzhamerstrasse\,23, 4600 Wels, Austria, \textsc{klaus.schiefermayr@fh-wels.at}}}
\date{}
\maketitle

\theoremstyle{plain}
\newtheorem{theorem}{Theorem}
\newtheorem{corollary}[theorem]{Corollary}
\newtheorem{lemma}[theorem]{Lemma}
\newtheorem{definition}[theorem]{Definition}
\newtheorem{notation}[theorem]{Notation}
\theoremstyle{definition}
\newtheorem*{remark}{Remark}
\newtheorem*{example}{Example}

\begin{abstract}
In this paper, some geometric properties of inverse polynomial images which consist of a real interval and an arc symmetric with respect to the real line are obtained. The proofs are based on properties of Jacobi's elliptic and theta functions.
\end{abstract}

\noindent\emph{Mathematics Subject Classification (2000):} 33E05, 30C10, 30C20

\noindent\emph{Keywords:} \emph{Inverse polynomial image}, \emph{Jacobi's elliptic functions}, \emph{Jacobi's theta functions}

\section{Introduction}


Let $\P_n$ be the set of all polynomials of degree $n$ and let $\T_n\in\P_n$. Let $\Tn$ be the inverse image of $[-1,1]$ under the polynomial mapping $\T_n$, i.e.,
\[
\Tn:=\bigl\{z\in\C:\T_n(z)\in[-1,1]\bigr\}.
\]
In general, $\Tn$ consists of $n$ Jordan arcs, on which $\T_n$ is strictly monotone increasing from $-1$ to $+1$, see \cite{Peh2003}. If there is a point $z_1\in\C$, for which $\T_n(z_1)\in\{-1,1\}$ and $\T_n'(z_1)=0$, then two Jordan arcs can be combined into one Jordan arc. This combination of arcs can be seen very clearly in the inverse image of the classical Chebyshev polynomial $T_n(z)=\cos(n\arccos(z))$. In this case, the inverse image $T_n^{-1}([-1,1])$ is just $[-1,1]$, i.e.\ \emph{one} Jordan arc, since there are $n-1$ points $z_j$ with the property $T_n(z_j)\in\{-1,1\}$ and $T_n'(z_j)=0$. Note that $T_n$ is, up to a linear transformation, the only polynomial mapping with this property. In \cite{PehSch}, Peherstorfer and the author have characterised the case in which $\Tn$ consists of \emph{two} Jordan arcs, see also \cite{Sch} for an algebraic solution of the problem. The paper in hand can be considered as a continuation of \cite{PehSch}, whereas here we focus on the case in which the two Jordan arcs consists of an interval and an arc symmetric with respect to the real line.


The paper is organised as follows: First, a general result on the number of extremal points on $\Tn$ is given. Then, we focus on the case in which $\Tn$ consists of a real interval and an arc symmetric with respect to the real line. We obtain several interesting geometric properties involving the connectivity of the two arcs, the number of extremal points of the corresponding polynomial on the two arcs and a density result concerning the endpoints of the arcs. Finally, some auxiliaries on Jacobis elliptic and theta functions, which are necessary for the proofs in Section\,2, are proved in Section\,3.

\section{Main results}


We call four pairwise distinct points $a_1,a_2,a_3,a_4\in\C$ a $\tn$-tuple if there exist polynomials $\T_n\in\P_n$ and $\U_{n-2}\in\P_{n-2}$ such that a polynomial equation (sometimes called Pell's equation) of the form
\begin{equation}\label{PolEqn}
\T_n^2(z)-\h(z)\,\U_{n-2}^2(z)=1
\end{equation}
holds, where
\[
\h(z):=(z-a_1)(z-a_2)(z-a_3)(z-a_4).
\]
Note that $\T_n$ and $\U_{n-2}$ are unique up to sign. It is proved in \cite[Thm.\,3]{PehSch} that the above polynomial equation \eqref{PolEqn} is equivalent to the fact that $\Tn$ consists of \emph{two Jordan arcs} with endpoints $a_1,a_2,a_3,a_4$. It will turn out that the point $z^*$ defined by the equation
\begin{equation}\label{dTn}
\frac{\dd}{\dd{z}}\Bigl\{\T_n(z)\Bigr\}=\pm{n}(z-z^*)\,\U_{n-2}(z)
\end{equation}
plays an important role in what follows. With the help of $z^*$, $\T_n$ can be given by the integral
\[
\T_n(z)=\pm\cosh\Bigl(n\int_{a_1}^{z}\frac{w-z^*}{\sqrt{\h(w)}}\,\dd{w}\Bigr).
\]
The first result gives the number of extremal points of $\T_n$ on $\Tn$. As usual, a point $z_0\in\Tn$ is called an extremal point of $\T_n$ on $\Tn$ if $|\T_n(z_0)|=\max_{z\in\Tn}|\T_n(z)|$, i.e., if $\T_n(z_0)\in\{-1,1\}$.


\begin{theorem}\label{Thm-EP}
Suppose that the polynomial equation \eqref{PolEqn} holds and let $z^*$ be defined by \eqref{dTn}. If $\T_n(z^*)\neq\pm1$ $\Bigl[\T_n(z^*)=\pm1\Bigr]$ then $\T_n$ has $n+2$ $\bigl[n+1\bigr]$ extremal points on $\Tn$.
\end{theorem}
\begin{proof}
\emph{Case\,1.}
$\T_n(z^*)\neq\pm1$. Then, by \cite[Lem.\,1(ii)]{PehSch}, $z^*$ is no zero of $\h\U_{n-2}$ and $\U_{n-2}$ has $n-2$ simple zeros which are no zeros of $\h$. Thus, by \eqref{dTn}, $\T_n^2-1$ has $n-2$ double zeros (zeros of $\U_{n-2}$) and 4 simple zeros (zeros of $\h$). This gives the assertion.\\
\emph{Case\,2.}
$\T_n(z^*)=\pm1$. Then, by \eqref{PolEqn}, $z^*$ is a zero of $\h\U_{n-2}$. By \cite[Lem.\,1(ii)]{PehSch}, there are two cases possible:\\
(i) $z^*$ is a double zero of $\U_{n-2}$ but not a zero of $\h$, and $\U_{n-2}$ has only simple zeros, apart from $z^*$.\\
(ii) $z^*$ is a zero of $\h$ and a simple zero of $\U_{n-2}$, and $\U_{n-2}$ has only simple zeros.\\
Thus, by \eqref{dTn}, in case (i), $\T_n^2-1$ has 4 simple zeros, $n-4$ double zeros and one zero of order 4, in case (ii), $\T_n^2-1$ has 3 simple zeros, $n-3$ double zeros and one zero of order 3. This completes the proof.
\end{proof}


In what follows, we focus on the case when $a_1,a_2\in\R$, $a_3,a_4\in\C$ with $a_3=\ov{a}_4$. By linear transformation and symmetry arguments, we only have to consider the case $a_1=-1$, $a_2=1$, $\re(a_3)>0$, $\im(a_3)>0$, $a_3=\ov{a}_4$. Since the characterisation of $\tn$-tuples involves a rigorous usage of Jacobi's elliptic and theta functions, let us briefly mention some notation.


Let $0<k<1$ denote the modulus of Jacobi's elliptic functions $\sn(u)$, $\cn(u)$, and $\dn(u)$, of Jacobi's theta functions $\Theta(u)$, $H(u)$, $H_1(u)$, and $\Theta_1(u)$, and, finally, of Jacobi's zeta function, $\zn(u)$. Let $k':=\sqrt{1-k^2}\in(0,1)$ be the complementary modulus, let $K\equiv{K}(k)$ and $E\equiv{E}(k)$ be the complete elliptic integral of the first and second kind, respectively, and let $K'\equiv K'(k):=K(k')$ and $E'\equiv{E}'(k):=E(k')$. For the definitions and many important properties
of these functions, see, e.g., \cite{ByrdFriedman,Lawden,MagnusOberhettingerSoni,WhittakerWatson}.


Let us consider the mapping
\begin{equation}\label{alphabeta}
\alpha+\ii\beta\equiv\alpha(k,\lambda)+\ii\beta(k,\lambda)
=\frac{\cn(2\laK)}{\dn^2(2\laK)}+\ii\,\frac{kk'\sn^2(2\laK)}{\dn^2(2\laK)}
\end{equation}
for $0<k<1$ and $0<\lambda<\frac{1}{2}$. This mapping is \emph{bijective} from $(0,1)\times(0,\frac{1}{2})$ onto $\bigl\{z\in\C:\re(z)>0,\im(z)>0\bigr\}$ (see Lemma\,\ref{Lemma-Bij}). By \cite[Theorem\,17]{PehSch}, for fixed $0<\lambda<\frac{1}{2}$, $\alpha(k,\lambda)$ is strictly monotone increasing in $k$ and $\alpha(k,\lambda)\to\cos(\lambda\pi)$, $\beta(k,\lambda)\to0$ as $k\to0$, and $\alpha(k,\lambda)\to+\infty$ as $k\to1$. For fixed $0<\lambda\leq\frac{1}{4}$, $\beta(k,\lambda)$ is also strictly monotone increasing in $k$. For fixed $0<\lambda<\frac{1}{4}$ and $\lambda=\frac{1}{4}$ and $\frac{1}{4}<\lambda<\frac{1}{2}$, $\beta(k,\lambda)\to0$ and $\beta(k,\lambda)\to1$ and $\beta(k,\lambda)\to+\infty$ as $k\to1$, respectively.

Moreover, the curves $\alpha(k,\lambda)+\ii\beta(k,\lambda)$ have the following nice property:


\begin{theorem}
Let $\alpha,\beta$ be defined in \eqref{alphabeta} then
\[
\alpha^2+\beta^2=1 \iff k=k'=\frac{1}{\sqrt{2}}.
\]
Further, $\alpha^2+\beta^2<1$ and $\alpha^2+\beta^2>1$ if and only if $k<\frac{1}{\sqrt{2}}$ and $k>\frac{1}{\sqrt{2}}$, respectively.
\end{theorem}
\begin{proof}
We have
\begin{align*}
&\alpha^2+\beta^2=1\\
&\iff \cn^2(2\laK)+k^2{k'}^2\sn^4(2\laK)=\dn^4(2\laK)\\
&\iff 1-\sn^2(2\laK)+k^2{k'}^2\sn^4(2\laK)=\bigl(1-k^2\sn^2(2\laK)\bigr)^2\\
&\iff (1-2k^2)\sn^2(2\laK)+k^2(k^2-{k'}^2)\sn^4(2\laK)=0\\
&\iff 1-2k^2=0 \quad\wedge\quad k^2-{k'}^2=0\\
&\iff k=\frac{1}{\sqrt{2}}
\end{align*}
The second assertion is clear since $\alpha+\ii\beta\to\cos(\lambda\pi)$ as $k\to0$ and $\alpha\to+\infty$ as $k\to1$.
\end{proof}


\noindent In Figure\,\ref{Fig_AlphaBeta}, the curves $\alpha(k,\lambda)+\ii\beta(k,\lambda)$ are plotted
\begin{itemize}
\item for fixed $\lambda\in\{\frac{1}{16},\frac{2}{16},\ldots,\frac{7}{16}\}$ and moving $k$ with solid lines (the curve for $\lambda=\frac{4}{16}=\frac{1}{4}$ is the fat solid line);
\item for fixed $k\in\{0.5,\frac{1}{\sqrt{2}},0.8\}$ and moving $\lambda$, $0<\lambda<\frac{1}{2}$, with dashed lines (the curve for $k=\frac{1}{\sqrt{2}}$ is the fat dashed line).
\end{itemize}


\begin{figure}[ht]
\begin{center}
\includegraphics[scale=1.4]{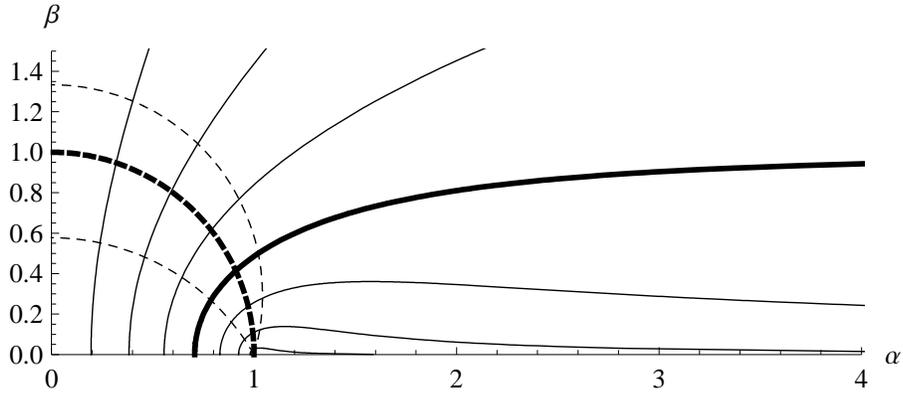}
\caption{\label{Fig_AlphaBeta} The curves $\alpha(k,\lambda)+\ii\beta(k,\lambda)$ for fixed $\lambda$ (solid lines) and fixed $k$ (dashing lines)}
\end{center}
\end{figure}


\noindent Let us summarize the known results on $\tn$-tuples which consists of an interval and an arc, see \cite[Section\,5.2]{PehSch}, \cite{Zolotarev2}, \cite[Section\,19]{Achieser1929} and \cite[pp.\,249--251]{Lebedev}.


\begin{theorem}\label{Thm_TnTuple}
Let $n\in\N$, and let $a_1=-1$, $a_2=1$, $a_3\in\C$, $\re(a_3)>0$, $\im(a_3)>0$, $a_4=\ov{a}_3$. Further, let $0<k<1$ and $0<\lambda<\frac{1}{2}$ be such that $a_3=\alpha(k,\lambda)+\ii\beta(k,\lambda)$, where $\alpha,\beta$ are defined in \eqref{alphabeta}.
\begin{enumerate}
\item The tuple $(a_1,a_2,a_3,a_4)=(-1,1,\alpha+\ii\beta,\alpha-\ii\beta)$ is a $\tn$-tuple if and only if $\lambda=\frac{m}{n}$, where $m\in\N$, i.e., if and only if $\lambda$ is rational with denominator $n$.
\item The associated polynomial $\T_n$ is given by
\begin{equation}\label{Tn}
\T_n(z(u))=\frac{1}{2}\Bigl(\Omega^n(u)+\frac{1}{\Omega^n(u)}\Bigr),
\end{equation}
where
\begin{equation}\label{Omega}
\Omega(u):=\frac{H(u-\laK)\,\Theta_1(u-\laK)}{H(u+\laK)\,\Theta_1(u+\laK)}
\end{equation}
and
\begin{equation}\label{z(u)}
z(u):=\frac{\cn(2u)\,\cn(2\laK)-1}{\cn(2u)-\cn(2\laK)}.
\end{equation}
\item The point $z^*$ defined by relation~\eqref{dTn} is given by
\begin{equation}\label{z*1}
z^*=1+\frac{1}{\dn(2\laK)}\bigl(\cn(2\laK)-1+2\,\sn(2\laK)\,\zn(\laK)\bigr)
\end{equation}
and $z^*>0$ holds.
\item The inverse image $\Tn$ is given by
\[
\Tn=\bigl\{z(u)\in\C: u\in\C, |\Omega(u)|=1\bigr\}
\]
and consists of the interval $[-1,1]$ and a Jordan arc symmetric with respect to the real line with the endpoints $a_3$ and $a_4$. We will denote this Jordan arc by $\arc$.
\item If $z^*\leq1$ then the arc $\arc$ crosses the interval $[-1,1]$ at $z^*$.
\item If $z^*>1$, then $\arc$ and $[-1,1]$ are non-intersecting.
\item For the special case $\lambda=\frac{1}{2}$, there is $a_3=\frac{\ii{k}}{k'}$, $z^*=0$,\\
$\T_n(z)=T_{\frac{n}{2}}(2{k'}^2(z^2-1)+1)$ and $\Tn=[-1,1]\cup[-\frac{\ii k}{k'},\frac{\ii{k}}{k'}]$.
\item If $\lambda$ changes to $1-\lambda$ then $\alpha$ changes to $-\alpha$, $\beta$ remains equal and $\Tn$ is reflected with respect to the imaginary axis.
\end{enumerate}
\end{theorem}


\begin{remark}
In Figure\,\ref{Fig_Arcs}, for $n=8$ and $\lambda=\frac{2}{8}$ (i.e.\ $m=2$), the inverse images $\T_n^{-1}(\R)$ (dotted line) and $\Tn$ (solid line) are plotted for the cases $z^*<1$ (picture above), $z^*=1$ (picture in the middle) and $z^*>1$ (picture below), respectively, see Theorem\,\ref{Thm_TnTuple}\,(v) and (vi). Note that the extremal points of $\T_n$ on $\Tn$ are the endpoints of the two arcs $[-1,1]$ and $\arc$ and the points of intersection of the dotted and the solid lines, see Theorem\,\ref{Thm-EP-Arcs}.
\end{remark}


\begin{figure}[ht]
\begin{center}
\includegraphics[scale=1.0]{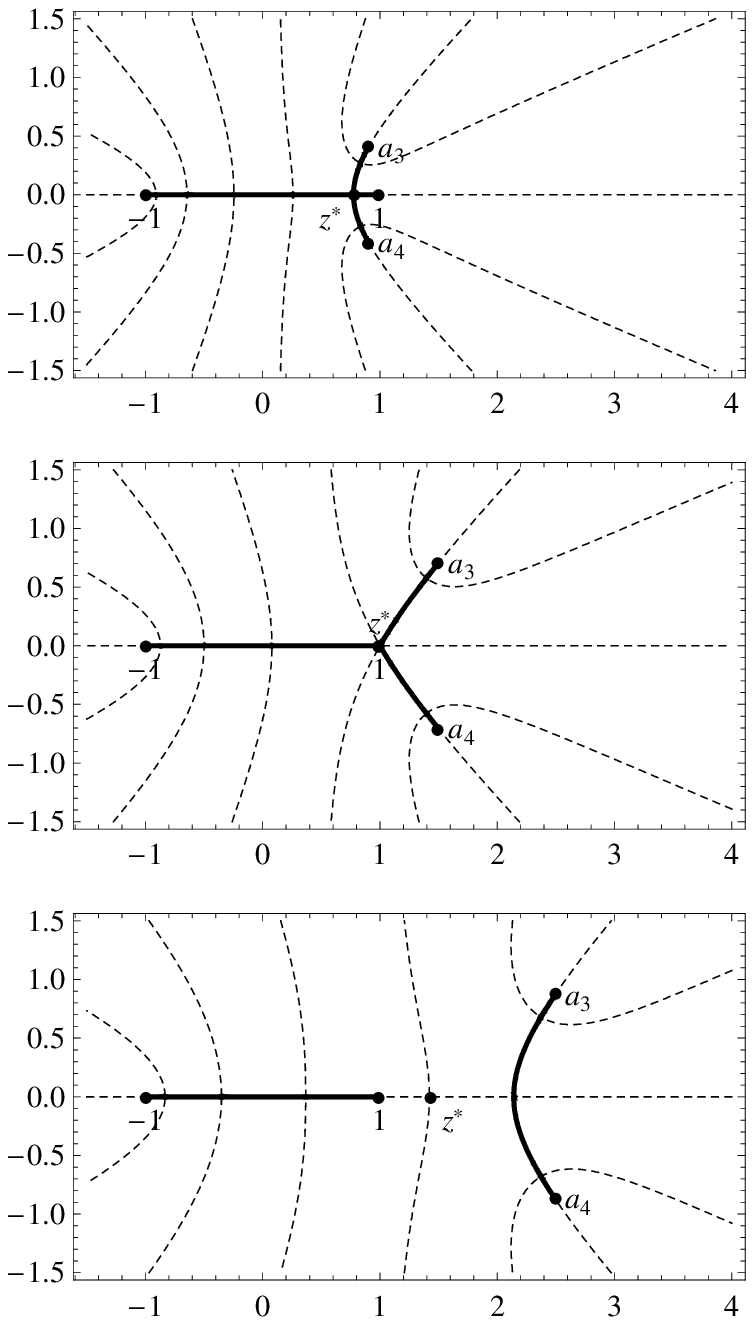}
\caption{\label{Fig_Arcs} Inverse images $\T_n^{-1}(\R)$ (dotted line) and $\Tn$ (solid line)}
\end{center}
\end{figure}


\begin{theorem}\label{Thm-EP-Arcs}
Let $0<k<1$, $0<\lambda<\frac{1}{2}$, $\lambda=\frac{m}{n}$, $m,n\in\N$, $0<m<\frac{n}{2}$, let $z(u)$ and $\Omega(u)$ be defined by \eqref{z(u)} and \eqref{Omega}, respectively, and let $\T_n(z(u))$ defined by \eqref{Tn}.
\begin{enumerate}
\item If $z^*\leq1$, i.e. the arcs $\arc$ and $[-1,1]$ are intersecting at the point $z^*$, then $\T_n(z)$ has at least $n-2m+1$ extremal points on $[-1,1]$.
\item If $z^*>1$, i.e. the arcs $\arc$ and $[-1,1]$ are non-intersecting, then $\T_n(z)$ has exactly $n-2m+1$ extremal points on $[-1,1]$ and exactly $2m+1$ extremal points on $\arc$.
\end{enumerate}
\end{theorem}
\begin{proof}
Recall that, by Lemma\,\ref{Lemma_z(u)}, $z(0)=-1$, $z(\iK)=1$, $z:[0,\iK]\to[-1,1]$, $u\mapsto{z}(u)$, is bijective, and, by Lemma\,\ref{Lemma_Omega}, $\Omega(0)=-1=\exp(\ii\pi)$ and $\Omega(\iK)=\exp(\frac{2\ii{m}\pi}{n})$. Since $[-1,1]\subseteq\Tn$, for $u\in[0,\iK]$, there is $|\Omega(u)|=1$, thus $\Omega(u)=\exp(\ii\varphi(u))$, where $\varphi(u)\in\R$. Hence, by \eqref{Tn}, for $u\in[0,\iK]$, $\T_n(z(u))=\cos(n\varphi(u))$, and $\T_n(z(u))$ has an extremal point iff $\varphi(u)=\nu\pi$, $\nu\in\Z$. Now, if $u$ moves from $0$ to $\iK$ then, for some $\ell\in\Z$, $\varphi(u)$ moves continuously from $\pi$ to $\frac{2m\pi}{n}+2\ell\pi$ and $n\varphi(u)$ moves continuously from $n\pi$ to $2m\pi+2\ell{n}\pi$. Thus the number of extremal points of $\T_n(z)$ on $[-1,1]$ is at least
\[
1+\min_{\ell\in\Z}\bigl|n-2m-2\ell{n}\bigr|=1+n-2m.
\]
So far, we have proved assertion (i) and a part of assertion (ii).\\
Let $z^*>1$ and let $\CC$ be that Jordan arc in the $u$-plane of Lemma\,\ref{Lemma_z(u)}\,(v). Recall that, by Lemma\,\ref{Lemma_z(u)}\,(iii), $z(\frac{K}{2}-\frac{\iK}{2})=a_3$ and $z(\frac{K}{2}+\frac{\iK}{2})=a_4$ and, by Lemma\,\ref{Lemma_Omega}, $\Omega(\tfrac{K}{2}\pm\tfrac{\iK}{2})=\exp(\pm\tfrac{\ii{m}\pi}{n})$. For $u\in\CC$, $|\Omega(u)|=1$, thus, as above, $\Omega(u)=\exp(\ii\varphi(u))$, where $\varphi(u)\in\R$, and $\T_n(z(u))=\cos(n\varphi(u))$. Now, if $u$ moves continuously from $\frac{K}{2}+\frac{\iK}{2}$ to $\frac{K}{2}-\frac{\iK}{2}$ then, for some $\ell\in\Z$, $\varphi(u)$ moves continuously from $\frac{m\pi}{n}$ to $-\frac{m\pi}{n}+2\ell\pi$ and $n\varphi(u)$ moves continuously from $m\pi$ to $-m\pi+2\ell{n}\pi$. Thus, the number of extremal points of $\T_n(z)$ on the arc $\arc$ is at least
\[
1+\min_{\ell\in\Z}\bigl|m-(-m+2\ell{n})\bigr|=1+2m.
\]
Taking into consideration Theorem\,\ref{Thm-EP}, assertion (ii) is completely proved.
\end{proof}


The next theorem gives a density result.


\begin{theorem}
Let $\alpha,\beta>0$ and $n\in\N$ be given. Then there exist $\alpha^*,\beta^*>0$ such that $(a_1,a_2,a_3,a_4)=(-1,1,\alpha^*+\ii\beta^*,\alpha^*-\ii\beta^*)$ is a $\tn$-tuple and
\[
|(\alpha+\ii\beta)-(\alpha^*+\ii\beta^*)|\leq\frac{A}{n}
\]
and $A$ is a constant independent of $n$.
\end{theorem}
\begin{proof}
Let $0<\lambda<\frac{1}{2}$ and $0<k<1$ be such that
\[
\alpha=\frac{\cn(2\laK)}{\dn^2(2\laK)},\qquad\beta=\frac{kk'\sn^2(2\laK)}{\dn^2(2\laK)},
\]
and let $m\in\{0,1,2,\ldots,\lfloor\frac{n}{2}\rfloor\}$ such that $|\lambda-\frac{m}{n}|\leq\frac{1}{n}$ and define
\[
\alpha^*:=\frac{\cn(2\frac{m}{n}K)}{\dn^2(2\frac{m}{n}K)},\qquad
\beta^*:=\frac{kk'\sn^2(2\frac{m}{n}K)}{\dn^2(2\frac{m}{n}K)}.
\]
Then
\begin{align*}
&|(\alpha+\ii\beta)-(\alpha^*+\ii\beta^*)|\\
&\leq\left|\frac{\cn(2\laK)}{\dn^2(2\laK)}-\frac{\cn(2\frac{m}{n}K)}{\dn^2(2\frac{m}{n}K)}\right|+
kk'\left|\frac{\sn^2(2\laK)}{\dn^2(2\laK)}-\frac{\sn^2(2\frac{m}{n}K)}{\dn^2(2\frac{m}{n}K)}\right|\\
&\leq\left|\lambda-\frac{m}{n}\right|\frac{2K}{{k'}^3}+kk'\left|\lambda-\frac{m}{n}\right|\frac{4K}{{k'}^3}
\qquad\text{by Lemma\,\ref{Lemma-Derivatives}}\\
&\leq\frac{1}{n}\Bigl(\frac{2K}{{k'}^3}+\frac{4kK}{{k'}^2}\Bigr)=:\frac{A}{n}
\end{align*}
\end{proof}


Next, we give some properties of the point $z^*$ as a function of the modulus $k$.


\begin{theorem}\label{Thm-z*}
The point $z^*$ given in \eqref{z*1} can also be expressed by the formula
\begin{equation}\label{z*2}
z^*=\cn(2\laK)+\frac{\sn(2\laK)\,\zn(2\laK)}{\dn(2\laK)}.
\end{equation}
For fixed $\lambda$, $0<\lambda<\frac{1}{2}$, the point $z^*$, as a function of the modulus $k$, $0<k<1$, is strictly monotone increasing. Furthermore,
\begin{equation}\label{z*3}
z^*\to\cos(\lambda\pi)\qquad\text{as $k\to0$}
\end{equation}
and
\begin{equation}\label{z*4}
z^*\sim\bigl(\frac{1}{2}-\lambda\bigr)\bigl(\frac{4}{k'}\bigr)^{2\lambda}\to\infty\qquad\text{as $k\to1$}.
\end{equation}
\end{theorem}
\begin{proof}
By the formulae (3.6.2) and (2.4.4) of \cite{Lawden}, we get
\[
\zn(u)=\frac{1}{2}\Bigl(\zn(2u)+k^2\sn(2u)\,\frac{1-\cn(2u)}{1+\dn(2u)}\Bigr)
\]
Applying this formula for $\zn(\laK)$ in \eqref{z*1} gives after some simplification \eqref{z*2}.\\
Using the formulae for the derivatives of $\sn(\laK)$, $\cn(\laK)$, $\dn(\laK)$, and $\zn(\laK)$ with respect to $k$ (for fixed $\lambda$), see \cite[Lem.\,39]{PehSch}, we get from \eqref{z*2}
\[
\frac{\dd}{\dd{k}}\bigl\{z^*\bigr\}=\frac{\zn(2\laK)\,\bigl[\sn(2\laK)\,\dn(2\laK)-\cn(2\laK)\,\zn(2\laK)\bigr]}{k{k'}^2\dn^2(2\laK)}.
\]
Thus, it remains to prove that
\begin{equation}\label{eqn1}
\sn(2\laK)\,\dn(2\laK)-\cn(2\laK)\,\zn(2\laK)>0.
\end{equation}
By Lemma\,\ref{Lemma_IneqZn},
\[
\zn(2\laK)\leq\frac{k^2\sn(2\laK)\,\cn(2\laK)}{\dn(2\laK)}<\frac{\sn(2\laK)\,\dn(2\laK)}{\cn(2\laK)},
\]
where the last inequality is true since
\[
0<\dn^2(2\laK)-k^2\cn^2(2\laK)={k'}^2.
\]
Hence, inequality \eqref{eqn1} is true and the monotonicity of $z^*$ is proved.\\
The limits of $z^*$ as $k\to0$ and $k\to1$ follow immediately from \cite{CarlsonTodd}.
\end{proof}


\begin{remark}
\begin{enumerate}
\item A similar formula as \eqref{z*2} can be found in \cite[eq.\,(54)]{Achieser1929}.
\item Formulae \eqref{z*3} and \eqref{z*4} can also be found in \cite[eq.\,(2.82),~eq.\,(2.83)]{Lebedev} (there is a misprint in formula (2.83)).
\item By Theorem\,\ref{Thm-z*}, for each $\lambda\in(0,\frac{1}{2})$, there exists a uniquely determined modulus $k^*\equiv{k}^*(\lambda)$, for which $z^*=1$, Thus, for $k\leq{k}^*$, the arc $\arc$ is intersecting the interval $[-1,1]$ (at $z^*$), and, for $k>k^*$, $\arc$ and $[-1,1]$ are non-intersecting. Compare Figure\,\ref{Fig_Arcs}, which shows the inverse images $\T_n^{-1}(\R)$ (dotted line) and $\Tn$ (solid line) for $n=8$, $\lambda=2/8$ and $k=0.7<k^*$ (picture above), $k^*=0.942809\ldots$ (picture in the middle) and $k=0.99>k^*$ (picture below), respectively.
\item Figure\,\ref{Fig_Bound} shows that parametric curve $\alpha(\lambda,k^*(\lambda))+\ii\beta(\lambda,k^*(\lambda))$, on which $z^*=1$, i.e. for all points $a_3=\alpha+\ii\beta$, $\alpha,\beta>0$, lying on the left hand side [right hand side] of this curve, the corresponding arc $\arc$ is intersecting [non-intersecting] the interval $[-1,1]$. The following theorem gives a simple sufficient condition such that the two arcs $\arc$ and $[-1,1]$ are intersecting each other.
\end{enumerate}
\end{remark}


\begin{figure}[ht]
\begin{center}
\includegraphics[scale=0.7]{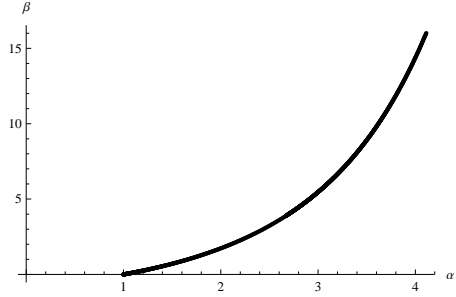}
\caption{\label{Fig_Bound} Parametric curve $\alpha(\lambda,k^*)+\ii\beta(\lambda,k^*)$, on which $z^*=1$}
\end{center}
\end{figure}


\begin{theorem}
If $\alpha\leq1$ the two arcs $\arc$ and $[-1,1]$ are intersecting each other.
\end{theorem}
\begin{proof}
Let $z^*$ be given by \eqref{z*2}. By Theorem\,\ref{Thm_TnTuple}\,(iv), if $z^*\leq1$ then $\arc$ is intersecting the interval $[-1,1]$. By inequality \eqref{IneqZn} of Lemma\,\ref{Lemma_IneqZn}, a sufficient condition for $z^*\leq1$ is
\begin{equation}\label{SuffCond}
\cn(2\laK)+\frac{k^2\sn^2(2\laK)\,\cn(2\laK)}{\dn^2(2\laK)}\leq1.
\end{equation}
By \eqref{alphabeta},
\begin{align*}
\eqref{SuffCond}&\iff\cn(2\laK)\Bigl[1+\frac{1-\dn^2(2\laK)}{\dn^2(2\laK)}\Bigr]\leq1\\
&\iff\frac{\cn(2\laK)}{\dn^2(2\laK)}\leq1\\
&\iff \alpha\leq1.
\end{align*}
This completes the proof.
\end{proof}

\section{Auxiliary results}


\begin{lemma}\label{Lemma-Bij}
The mapping
\begin{equation}
\begin{array}{rccc}
f:&(0,1)\times(0,\frac{1}{2})&\to&\R^+\times\R^+\\
&(k,\lambda)&\mapsto&\Bigl(\dfrac{\cn(2\laK)}{\dn^2(2\laK)},\dfrac{kk'\sn^2(2\laK)}{\dn^2(2\laK)}\Bigr)
\end{array}
\end{equation}
is bijective.
\end{lemma}
\begin{proof}
The system of equations
\begin{align*}
\alpha&=\frac{\cn(2\laK)}{\dn^2(2\laK)}\\
\beta&=\frac{kk'\sn^2(2\laK)}{\dn^2(2\laK)}
\end{align*}
is equivalent to (using the formulas $k^2\sn^2(u)=1-\dn^2(u)$ and ${k'}^2\sn^2(u)=\dn^2(u)-\cn^2(u)$)
\begin{align*}
\alpha^2&=\frac{\cn^2(2\laK)}{\dn^4(2\laK)}\\
\beta^2&=\frac{(1-\dn^2(2\laK))(1-\alpha^2\dn^2(2\laK))}{\dn^2(2\laK)}
\end{align*}
which is again equivalent to (using the formula $k^2=(1-\dn^2(u))/(1-\cn^2(u))$)
\begin{align}
\dn^2(2\laK)&=\frac{1}{2\alpha^2}\bigl(1+\alpha^2+\beta^2-\sqrt{(1+\alpha^2+\beta^2)^2-4\alpha^2}\bigr)\label{eqn-dn}\\
k^2&=\frac{1-\dn^2(2\laK)}{1-\alpha^2\dn^4(2\laK)}\label{eqn-k2}
\end{align}
Thus, for given $\alpha,\beta>0$, in a unique way, one can determine $\dn^2(2\laK)$ from \eqref{eqn-dn}, then $k$ from \eqref{eqn-k2} and finally $0<\lambda<\frac{1}{2}$ with the help of \eqref{eqn-dn}.
\end{proof}


\begin{lemma}\label{Lemma_Ineq}
Let $f\in{C}^2([a,b])$, $a<b$. If $f(a)=f(b)=0$ and $f''(u)\geq0$ $\left[f''(u)\leq0\right]$ for $u\in[a,b]$ then $f(u)\leq0$ $\left[f(u)\geq0\right]$ for $u\in[a,b]$.
\end{lemma}


\begin{lemma}\label{Lemma_IneqZn}
Let $0\leq{k}<1$, then for $0\leq{u}\leq{K}$,
\begin{equation}\label{IneqZn}
\frac{k^2{k'}^2}{1+{k'}^2}\cdot\frac{\sn(u)\,\cn(u)}{\dn(u)}
\leq\zn(u)\leq\frac{k^2}{1+{k'}^2}\cdot\frac{\sn(u)\,\cn(u)}{\dn(u)},
\end{equation}
where equality is attained for $u=0$ or $u=K$ or $k=0$.
\end{lemma}
\begin{proof}
Let
\[
f(u):=\zn(u)-B\cdot\frac{\sn(u)\,\cn(u)}{\dn(u)},
\]
then $f(0)=f(K)=0$ and
\begin{align*}
f''(u)&=-2k^2\sn(u)\,\cn(u)\,\dn(u)+B\cdot\frac{2\,\sn(u)\,\cn(u)}{\dn^3(u)}\cdot\bigl({k'}^2+\dn^4(u)\bigr)\\
&=\frac{2\,\sn(u)\,\cn(u)}{\dn^3(u)}\Bigl(-k^2\dn^4(u)+B\bigl({k'}^2+\dn^4(u)\bigr)\Bigl)
\end{align*}
For $B:=k^2/(1+{k'}^2)$ and $B:=k^2{k'}^2/(1+{k'}^2)$, by the trivial inequality $k'\leq\dn(u)\leq1$, we get $f''(u)\geq0$ and $f''(u)\leq0$, respectively. By Lemma\,\ref{Lemma_Ineq}, both inequalities of \eqref{IneqZn} are proved.
\end{proof}


Using the basic properties of the function $\cn(2u)$, it is easy to derive successively the following properties for the function $z(u)$ defined in \eqref{z(u)}:


\begin{lemma}\label{Lemma_z(u)}
Let $0<k<1$, $0<\lambda<\frac{1}{2}$, then $z(u)$, defined by \eqref{z(u)}, has the following properties:
\begin{enumerate}
\item $z(-u)=z(u)$, $z(\ov{u})=\ov{z(u)}$
\item $z(u)$ is an elliptic function of order 2 with fundamental periods $2K$ and $K+\iK$ and simple poles at $\pm\laK$.
\item $z(0)=z(K+\iK)=-1$, $z(\iK)=z(K)=1$, $z(\laK)=z(K-\laK+\iK)=\infty$, $z(\frac{K}{2}-\frac{\iK}{2})=a_3$, $z(\frac{K}{2}+\frac{\iK}{2})=a_4$
\item The following mappings $z:A\to{B}$, $u\mapsto{z}(u)$, are bijective:\\
$[0,\iK]\to[-1,1]$, $[0,\laK)\to(-\infty,-1]$, $(\laK,K]\to[1,\infty)$, $[K,K+\iK]\to[-1,1]$, $[\iK,K-\laK+\iK)\to[1,\infty)$, $(K-\laK+\iK,K+\iK]\to(-\infty,-1]$
\item If $z^*>1$ then there exists a continuous Jordan arc $\CC$ in the $u$-plane with endpoints $\frac{K}{2}+\frac{\iK}{2}$ and $\frac{K}{2}-\frac{\iK}{2}$, crossing the real line at a point of the interval $(\laK,K)$, such that for all $u\in\CC$ there is $|\Omega(u)|=1$.
\end{enumerate}
\end{lemma}


\begin{lemma}\label{Lemma_Omega}
Let $0<k<1$, $0<\lambda<\frac{1}{2}$, $\lambda=\frac{m}{n}$, $m,n\in\N$, $0<m<\frac{n}{2}$, and let $\Omega(u)$ be defined by \eqref{Omega}. Then
\[
\Omega(0)=-1,\quad\Omega(\iK)=\exp(\tfrac{2\ii{m}\pi}{n}),\quad
\Omega(\tfrac{K}{2}\pm\tfrac{\iK}{2})=\exp(\pm\tfrac{\ii{m}\pi}{n}).
\]
\end{lemma}
\begin{proof}
First let us note that $H(u)$ is an odd function of $u$ and $H_1(u)$, $\Theta_1(u)$, and $\Theta(u)$ are even functions of $u$. Therefore, $\Omega(0)=-1$ follows immediately. By the formulae
\begin{align*}
H(u+\iK)&=\ii\,q^{-1/4}\exp(-\tfrac{\ii\pi{u}}{2K})\,\Theta(u),\\
\Theta_1(u+\iK)&=q^{-1/4}\exp(-\tfrac{\ii\pi{u}}{2K})\,H_1(u),
\end{align*}
we get
\[
\Omega(\iK)=\exp(\tfrac{2\ii\pi\laK}{K})=\exp(\tfrac{2\ii{m}\pi}{n}).
\]
By the formula
\[
H(u+K+\iK)=H_1(u+\iK)=q^{-1/4}\,\exp(-\tfrac{\ii\pi{u}}{2K})\,\Theta_1(u)
\]
we get
\begin{align*}
\Omega(\tfrac{K}{2}+\tfrac{\iK}{2})&=\frac{H(-\tfrac{K}{2}-\tfrac{\iK}{2}-\laK+K+\iK)}
{H(-\tfrac{K}{2}-\tfrac{\iK}{2}+\laK+K+\iK)}\cdot\frac{\Theta_1(\tfrac{K}{2}+\tfrac{\iK}{2}-\laK)}
{\Theta_1(\tfrac{K}{2}+\tfrac{\iK}{2}+\laK)}\\
&=\frac{q^{-1/4}\exp\bigl(-\frac{\ii\pi}{2K}(-\frac{K}{2}-\frac{\iK}{2}-\laK)\bigr)}
{q^{-1/4}\exp\bigl(-\frac{\ii\pi}{2K}(-\frac{K}{2}-\frac{\iK}{2}+\laK)\bigr)}\\
&\qquad\cdot\frac{\Theta_1(-\tfrac{K}{2}-\tfrac{\iK}{2}-\laK)}{\Theta_1(-\tfrac{K}{2}-\tfrac{\iK}{2}+\laK)}\cdot
\frac{\Theta_1(\tfrac{K}{2}+\tfrac{\iK}{2}-\laK)}{\Theta_1(\tfrac{K}{2}+\tfrac{\iK}{2}+\laK)}\\
&=\exp(\tfrac{\ii\pi\laK}{K})=\exp(\tfrac{\ii{m}\pi}{n}).
\end{align*}
\end{proof}


\begin{lemma}\label{Lemma-Derivatives}
For $0<k<1$ and $0<\mu<\frac{1}{2}$, we have
\[
\frac{\dd}{\dd\mu}\Bigl\{\frac{\cn(2\mu{K})}{\dn^2(2\mu{K})}\Bigr\}
=\frac{2K\sn(2\mu{K})}{\dn^3(2\mu{K})}\Bigl[k^2\cn(2\mu{K})+{k'}^2\Bigr]
\leq\frac{2K}{{k'}^3}
\]
and
\[
\frac{\dd}{\dd\mu}\Bigl\{\frac{\sn^2(2\mu{K})}{\dn^2(2\mu{K})}\Bigr\}
=\frac{4K\sn(2\mu{K})\cn(2\mu{K})}{\dn^3(2\mu{K})}\leq\frac{4K}{{k'}^3}
\]
\end{lemma}
\begin{proof}
The derivatives are obtained by the formulae (731.01)--(731.03) of \cite{ByrdFriedman}, the inequalities are a direct consequence of the trivial estimates $0<\sn(u)<1$, $0<\cn(u)<1$, and $k'<\dn(u)<1$, for $0<u<K$.
\end{proof}


\bibliographystyle{amsplain}

\bibliography{IntervalArc}

\providecommand{\bysame}{\leavevmode\hbox to3em{\hrulefill}\thinspace}
\providecommand{\MR}{\relax\ifhmode\unskip\space\fi MR }
\providecommand{\MRhref}[2]{%
  \href{http://www.ams.org/mathscinet-getitem?mr=#1}{#2}
}
\providecommand{\href}[2]{#2}
\begin{thebibliography}{10}

\bibitem{Achieser1929}
N.I. Achieser, \emph{{\"U}ber einige {F}unktionen, die in gegebenen
  {I}ntervallen am wenigsten von {N}ull abweichen}, Bull. Phys. Math.
  \textbf{3} (1929), 1--69 (in German).

\bibitem{ByrdFriedman}
P.F. Byrd and M.D. Friedman, \emph{Handbook of elliptic integrals for engineers
  and scientists}, Springer, 1971.

\bibitem{CarlsonTodd}
B.C. Carlson and J.~Todd, \emph{The degenerating behavior of elliptic
  functions}, SIAM J. Numer. Anal. \textbf{20} (1983), 1120--1129.

\bibitem{Lawden}
D.F. Lawden, \emph{Elliptic functions and applications}, Springer, 1989.

\bibitem{Lebedev}
V.I. Lebedev, \emph{Zolotarev polynomials and extremum problems}, Russian J.
  Numer. Anal. Math. Modelling \textbf{9} (1994), 231--263.

\bibitem{MagnusOberhettingerSoni}
W.~Magnus, F.~Oberhettinger, and R.P. Soni, \emph{Formulas and theorems for the
  special functions of mathematical physics}, Springer, 1966.

\bibitem{Peh2003}
F.~Peherstorfer, \emph{Inverse images of polynomial mappings and polynomials
  orthogonal on them}, J. Comput. Appl. Math. \textbf{153} (2003), 371--385.

\bibitem{PehSch}
F.~Peherstorfer and K.~Schiefermayr, \emph{Description of inverse polynomial
  images which consist of two {J}ordan arcs with the help of {J}acobi's
  elliptic functions}, Comput. Methods Funct. Theory \textbf{4} (2004),
  355--390.

\bibitem{Sch}
K.~Schiefermayr, \emph{Inverse polynomial images which consists of two {J}ordan
  arcs -- {A}n algebraic solution}, J. Approx. Theory \textbf{148} (2007),
  148--157.

\bibitem{WhittakerWatson}
E.T. Whittaker and G.N. Watson, \emph{A course of modern analysis}, Cambridge
  University Press, 1962.

\bibitem{Zolotarev2}
E.I. Zolotarev, \emph{Applications of elliptic functions to problems of
  functions deviating least and most from zero (in {R}ussian)}, Oeuvres de E.I.
  Zolotarev, Vol.\,2, Izdat. Akad. Nauk SSSR, Leningrad (1932),
  1--59.\\Available on http://www.math.technion.ac.il/hat/papers.html.

\end{thebibliography}

\end{document}